\documentclass[11pt]{amsart}
\usepackage{amsmath,amsthm,hyperref,mathtools} 
\usepackage[alphabetic,initials,nobysame]{amsrefs}
\usepackage{color,nicefrac}
\usepackage[shortlabels]{enumitem}
\renewcommand{\MR}[1]{}

\title{Length distortion of volume-preserving Lipschitz mappings}

\author{Damaris Meier}
\address[Damaris Meier]{Department of Mathematics\\ ETH Zurich \\ R\"amistrasse 101\\ 8092 Zurich, Switzerland}
\email{damaris.meier@math.ethz.ch}

\author{Kai Rajala}
\address[Kai Rajala]{Department of Mathematics and Statistics, University of Jyväskylä, P.O. Box 35, 40014 University of Jyväskylä, Finland}
\email{kai.i.rajala@jyu.fi}


\date{\today}

\numberwithin{equation}{section}

\newtheorem{thm}{Theorem}[section]
\newtheorem{prop}[thm]{Proposition}
\newtheorem{question}[thm]{Question}
\newtheorem{lemma}[thm]{Lemma}
\newtheorem{corollary}[thm]{Corollary}

\theoremstyle{remark}
\newtheorem{rmk}[thm]{Remark}

\theoremstyle{definition}
\newtheorem{example}[thm]{Example}
\newtheorem{definition}[thm]{Definition}

\newcommand{\R}{\mathbb{R}}
\newcommand{\N}{\mathbb{N}}
\newcommand{\Nloc}{N^{1,2}_{\text{loc}}}

\newcommand{\hm}{{\mathcal H}}
\newcommand{\diam}{\operatorname{diam}}

\renewcommand{\mod}{\operatorname{Mod}}

\DeclareMathOperator{\md}{md}
\DeclareMathOperator{\apmd}{apmd}

\begin{document}
\begin{abstract}
We show that every area-preserving Lipschitz map between metric surfaces distorts the length of almost every curve by at most a multiplicative factor, answering a question posed by the first author and Ntalampekos. We give two different proofs. The first is entirely intrinsic and extends to higher dimensions under additional assumptions. The second relies on non-smooth uniformization theory and yields the conclusion that the family of curves intersecting the purely 2-unrectifiable part of a metric surface in a set of positive length is exceptional. 
\end{abstract}

\maketitle

\section{Introduction}
The Lipschitz-volume rigidity problem poses the question of whether any surjective $1$-Lipschitz map between metric spaces, which share the same volume (e.g.\ arising from Hausdorff measure), must be an isometry, which is a distance-preserving map. It is well-known that the answer to the Lipschitz-volume rigidity problem is affirmative if both spaces are closed Riemannian $n$-manifolds, see \cite{BuragoIvanov} and \cite{BessonEtAl}. In recent years there has been a growing interest in finding analogs for non-smooth settings, see e.g.\ \cites{Storm,Li:survey,CreutzSoultanis,BassoCreutzSoultanis,DelNinPerales,Zuest,MeierNtalampekos,BassoMartiWenger,Marti}. 

Every surjective 1-Lipschitz map $f\colon X\to Y$ between metric spaces $X$ and $Y$ of the same volume $\hm_X^n(X)=\hm_Y^n(Y)<\infty$ satisfies $\hm_X^n(A)=\hm_Y^n(f(A))$ for every measurable set $A\subset X$. A map with this property is called \emph{volume-preserving} (or \emph{area-preserving} for $n=2$). A possible generalization of the Lipschitz-volume rigidity problem is the following question.

\begin{question}\label{quest:main}
    Let $f\colon X\to Y$ be a volume-preserving Lipschitz map between metric spaces of locally finite Hausdorff $n$-measure. Which assumptions on $X$ and/or $Y$ ensure that $f$ satisfies additional regularity properties?
\end{question}

Here and below, \emph{locally finite} means finite on compact subsets. Recently, the first author and Ntalampekos \cite{MeierNtalampekos} established several results concerning Question \ref{quest:main} for maps between \emph{metric surfaces} $X$ and $Y$  satisfying additional assumptions, see \cite{MeierNtalampekos}*{Theorem 1.4}.

\begin{definition}
We call a metric space $X$ of locally finite Hausdorff $2$-measure a \emph{metric surface} if $X$ is homeomorphic to a connected $2$-manifold without boundary.  
\end{definition}

Question \eqref{quest:main} for general metric surfaces has remained open, see \cite{MeierNtalampekos}*{Question~1.5}. It is answered by the main result of this work, Theorem \ref{thm:main}, using the following notion of exceptional curve families.

\begin{definition}\label{def:exceptional}
We say that a family $\Gamma_0$ of rectifiable curves in a metric space $X$ with locally finite Hausdorff $n$-measure is \emph{exceptional} if there is a set $A \subset X$, $\hm^n(A)=0$, such that $\hm^1(|\gamma|\cap A)>0$ for every $\gamma \in \Gamma_0$. A property holds for \emph{almost every} curve in $X$ if the family of curves in $X$ for which the property fails is exceptional. 
\end{definition}
It follows from the definitions that the \emph{$p$-modulus} of an exceptional curve family is zero for every $1 \leq p < \infty$, see Section \ref{section:modulus}.

\begin{thm}\label{thm:main}
    Let $X$ and $Y$ be metric surfaces and $f\colon X\to Y$ an area-preserving $L$-Lipschitz map, $L\ge1$. Then there exists $C=C(L)>0$ such that 
    \begin{equation}\label{ineq:BLD}
        C\,\ell(\gamma)\leq\ell(f\circ\gamma)\leq L\,\ell(\gamma)
    \end{equation}
    holds for almost every curve $\gamma$ in $X$.
\end{thm}

We may set $C(L)=\frac{\pi}{8L}$ and $C(1)=1$, see Theorem \ref{thm:rectifiable} and Remark \ref{rmk:C(1)=1} below. The constant $C(L)=\frac{\pi}{8L}$ is not optimal, and it would be interesting to know how much it can be improved. The fact that inequality \eqref{ineq:BLD} is only satisfied for almost every curve in $X$ is sharp, as illustrated by the following example, compare to \cite{MeierNtalampekos}*{Example~4.1}. 

\begin{example}
    Denote by $I$ the unit interval embedded isometrically in $\R^2$, and let $Y:=\R^2/I$ be equipped with the quotient metric. Then the natural projection $f\colon\R^2\to Y$ is $1$-Lipschitz and area-preserving, and $f$ collapses any curve parametrizing a subinterval of $I$ to a curve of length $0$.
\end{example}

A map $f\colon X\to Y$ satisfying inequality \eqref{ineq:BLD} for every curve $\gamma$ in $X$ is called of \emph{bounded length distortion (BLD)}, a notion introduced in \cite{MartioVaisala}. It is shown in \cite{MeierNtalampekos}*{Theorem~1.4} that, under additional assumptions on $Y$, the map $f$ upgrades to being a quasiconformal homeomorphism that is also BLD. Moreover, if $Y$ is Riemannian, then a 1-Lipschitz area-preserving map $f\colon X\to Y$ is an isometry, providing a Lipschitz-volume rigidity result with no assumptions on the domain, see \cite{MeierNtalampekos}*{Theorem 1.4} and \cite{BassoMartiWenger}*{Theorem 1.4}. In this work, we present two different proofs of Theorem \ref{thm:main}.

\subsection{Intrinsic proof of Theorem \ref{thm:main} and higher dimensions}
The first proof of Theorem \ref{thm:main} is entirely intrinsic, in contrast to the work \cite{MeierNtalampekos}, which is heavily dependent on a deep uniformization theorem of metric surfaces \cite{NR22}, see also \cites{NR:21,MW21}. For almost every rectifiable curve $\gamma$ in $X$ parametrized by arc length, we will show the following. 

First, we use planar topological results and the coarea inequality for Lipschitz functions to establish a \emph{quadratic lower bound on the Hausdorff 2-measure} of small balls centered on the image $|\gamma|$ close to $x=\gamma(t)$, where $t$ is a point of \emph{strong metric differentiability} of $\gamma$. We refer to Section \ref{sec:metric-diff} for the definition of strong metric differentiability and to Section \ref{sec:auxiliary} for the exact statement of the result and its proof. 

The area-preservation and Lipschitz property of $f$ can then be used to prove a uniform lower bound on the metric derivative $|(f\circ\gamma)'|(t)$ of $f\circ\gamma$ at $t$ whenever $f(x)$ is a point of \emph{bounded upper density}. Indeed, if the lower bound does not hold, then, by the Lipschitz property, the images of small balls centered at $|\gamma|$ close to $x$ are contained in a small ball $B$ centered at $f(x)$. We may now estimate $\hm^2(B)$ from below by making use of the quadratic lower bound mentioned above and the area-preservation of $f$. By letting all radii go to 0, we eventually get a contradiction to the upper bound on the density at $f(x)$. The lower bound on $|(f\circ\gamma)'|(t)$ then gives the desired length distortion estimate. We refer to Section \ref{sec:proof:intrinsic} for more details. 

The planar structure of a metric surface is only used in the proof of the quadratic lower bound on $\hm^2$. Thus, the proof strategy extends to higher dimensions under the \emph{lower Ahlfors $n$-regularity assumption} \eqref{eq:lower-Ahlfors} as follows.

\begin{thm}\label{thm:main-higher-dim}
    Let $X$ and $Y$ be metric spaces of locally finite Hausdorff $n$-measure, $n \geq 2$, and assume that 
    there exists $c>0$ such that 
\begin{equation}\label{eq:lower-Ahlfors}
    \hm^n(B(x,r))\ge c\, r^n \quad \text{for every } x \in X \text{ and } 
    0<r<\diam(X). 
\end{equation}
    If $f\colon X\to Y$ is a volume-preserving $L$-Lipschitz map, $L\ge1$, then there is $0<C=C(L,c)\leq 1$ such that
    \begin{equation*}
        C\,\ell(\gamma)\leq\ell(f\circ\gamma)\leq L\,\ell(\gamma)
    \end{equation*}
    holds for almost every curve $\gamma$ in $X$.
\end{thm}

The class of lower Ahlfors $n$-regular spaces contains all closed \emph{linearly locally contractible metric $n$-manifolds}, see \cite{BassoMartiWenger}*{Theorem~4.1}. It is desirable to find conditions weaker than \eqref{eq:lower-Ahlfors} under which the conclusion of Theorem \ref{thm:main-higher-dim} remains true. 
This is the case for $L=1$ after replacing \eqref{eq:lower-Ahlfors} by countable $n$-rectifiability of $X$, see \cite{CreutzSoultanis}*{Proposition 4.1}. 

\begin{question} 
Does Theorem \ref{thm:main-higher-dim} hold if, instead of \eqref{eq:lower-Ahlfors}, we assume that $X$ is countably $n$-rectifiable?
\end{question}

\subsection{Proof of Theorem \ref{thm:main} via uniformization of metric surfaces}
Our second proof of Theorem \ref{thm:main} depends on the above mentioned uniformization result \cite{NR22}. On the way, we gain new insights into the structure of the ``singular'' part of a metric surface; see Theorem \ref{thm:unrectifiable} below. 

By \cite{NR22}, for any metric surface $X$ there are a Riemannian 2-manifold $(M,g)$ of constant curvature that is homeomorphic to $X$ and a continuous, surjective, proper, monotone and \emph{weakly quasiconformal} Sobolev map $u\in N^{1,2}_\text{loc}(M,X)$, see Section \ref{section:uniformization}. We refer to Section \ref{section:Sobolev} for a definition of a Sobolev map $u\in N^{1,2}_\text{loc}(M,X)$ and only note that, up to a set $G_0\subset M$ of Hausdorff 2-measure zero, the domain $M$ may be exhausted by sets on which $u$ is Lipschitz, see Section \ref{section:uniformization}. We set $X^u_0:=u(G_0)$ and note that $X\setminus X^u_0$ is countably $2$-rectifiable. We refer to $X\setminus X^u_0$ as the \emph{rectifiable image} of $u$ and to $X^u_0$ as the \emph{singular-set image} of $u$. 

The conclusion of Theorem \ref{thm:main} for curves whose intersections with the singular-set image of $u$ are negligible follows from an application of \emph{approximate metric differentiability} of Sobolev maps with euclidean domains, and the associated area formula and dilatation bounds, see Section \ref{section:diff} and Theorem \ref{thm:rectifiable} below. Therefore, Theorem \ref{thm:main} is derived from the next result. 

\begin{thm}\label{thm:unrectifiable}
    Let $X$ be a metric surface. The family of all curves $\gamma$ in $X$ satisfying $\hm^1(|\gamma|\cap X^u_0)>0$ is exceptional for any choice of the uniformization map $u$. 
\end{thm}

In addition to Theorem \ref{thm:main}, Theorem \ref{thm:unrectifiable} 
also yields results of independent interest on the structure of a metric surface. The following statement is a direct consequence of Theorem \ref{thm:unrectifiable}. For a definition of a purely 2-unrectifiable part of a metric surface, we refer to Section \ref{section:Hausdorff}.

\begin{corollary}\label{cor:unrectifiable}
    Let $X'$ be a purely 2-unrectifiable part of a metric surface $X$. Then the family of all curves $\gamma$ in $X$ satisfying $\hm^1(|\gamma|\cap X')>0$ is exceptional. 
\end{corollary} 

Finally, the proof of Theorem \ref{thm:unrectifiable} can be used to show that the singular-set image of $u$ is essentially independent of $u$. 

\begin{thm} \label{thm:singular} 
Let $X$ be a metric surface. There exists $X_0 \subset X$ such that for any choice of the uniformization map $u$, we have 
$$
\hm^2(X_0 \setminus X^u_0)=\hm^2(X^u_0 \setminus X_0)=0. 
$$ 
\end{thm} 
We call $X_0$ the \emph{singular part} of $X$, and note that $X_0$ is unique up to a set of Hausdorff $2$-measure zero. 

From the definitions it follows that $\hm^2(X' \setminus X_0)=0$ for every metric surface $X$. The complementary property $\hm^2(X_0 \setminus X')=0$ is not always true: one can show that if $X$ is the countably $2$-rectifiable metric surface constructed in \cite{EIR22}*{Theorem 5.3}, then $\hm^2(X_0)>0$. Theorems \ref{thm:unrectifiable} and \ref{thm:singular} and Corollary \ref{cor:unrectifiable} show that although the singular and purely 2-unrectifiable parts of a metric surface $X$ may have positive area (see e.g.\ \cite{R17}*{Proposition~17.1}), they are nevertheless negligible from the perspective of curve families and modulus.

\medskip

\textbf{Structure of the article:} In Section \ref{section:prel} we collect the necessary definitions and basic statements. Section \ref{sec:auxiliary} is devoted to establishing a quadratic lower bound on area in metric surfaces. In Section \ref{sec:proof:intrinsic} we present the intrinsic proofs of Theorems \ref{thm:main} and \ref{thm:main-higher-dim}. In Section \ref{sec:proof-thm-unrectifiable} we provide a proof of Theorems~\ref{thm:unrectifiable} and \ref{thm:singular}. The proof of Theorem \ref{thm:main} through uniformization of metric surfaces is presented in Section \ref{sec:proof-uniformization}.

\medskip

A reader only interested in the intrinsic proofs of Theorems \ref{thm:main} and \ref{thm:main-higher-dim} may skip Sections \ref{section:modulus}, \ref{section:Sobolev}, \ref{section:diff}, 
\ref{section:uniformization}, as well as Sections \ref{sec:proof-thm-unrectifiable} and \ref{sec:proof-uniformization}.

\medskip

\textbf{Acknowledgments:} This work was partially supported by the Simons Foundation grant (award no.\ SFI-MPS-T-Institutes-00010825) and from State Treasury funds as part of a task commissioned by the Minister of Science and Higher Education under the project “Organization of the Simons Semesters at the Banach Center - New Energies in 2026-2028” (agreement no.\ MNiSW/2025/DAP/491). The second author was supported by the Research Council of Finland, Project number 360505. Parts of this research were conducted when the first author was visiting University of Jyväskylä. She wishes to thank the department for their friendliness and hospitality.  

We thank Urs Lang, Denis Marti, Dimitrios Ntalampekos, Matthew Romney and Elefterios Soultanis for comments on an earlier version of this manuscript. 
\section{Preliminaries} \label{section:prel}

\subsection{Basic definitions and notations}
Let $(X,d)$ be a metric space. We denote the open (resp., closed) ball in $X$ of radius $r>0$ centered at a point $x\in X$ by $B(x,r)$ (resp., $\bar{B}(x,r)$). The \emph{open unit disk} refers to $D:=B(0,1)\subset\R^2$.
The \emph{metric sphere} in $X$ of radius $r>0$ centered at $x\in X$ is the set $S(x,r):=\{y\in X:d(x,y)=r\}$.

A \emph{curve} is a continuous map $\gamma\colon I\to X$, where $I\subset\R$ is an interval. The image is indicated by $|\gamma|=\gamma(I)$. The \textit{length} of $\gamma$ is its total variation and is denoted by $\ell(\gamma)$. A curve $\gamma$ is \emph{rectifiable} if $\ell(\gamma)<\infty$ and \emph{locally rectifiable} if each of its compact subcurves is rectifiable. 
A curve $\gamma\colon I\to X$ is \emph{parametrized by arc length} if $\ell(\gamma|_J)=|J|_1$ for every interval $J\subset I$. Here and later on, $|\cdot|_n$ denotes the \emph{$n$-dimensional Lebesgue measure}.

\subsection{Hausdorff measures}\label{section:Hausdorff}
For a metric space $X$ and $n\in\N$, the \textit{Hausdorff $n$-measure} of a set $A \subset X$ is defined by
$$\mathcal{H}^n(A) = \lim_{\delta\to 0} \mathcal H^n_{\delta}(A),\,\, \textrm{where}\,\,\,\, \mathcal H^n_\delta(A) =\inf \left\{ \sum_{j=1}^\infty \frac{\omega_n}{2^n} \diam(A_j)^n\right\} $$
and the infimum is taken over all collections of sets $\{A_j\}_{j=1}^\infty$ such that $A \subset \bigcup_{j=1}^\infty A_j$ and $\diam(A_j) < \delta$ for each $j$. Here $\omega_n$ is a positive normalization constant, chosen so that $|U|_n=\mathcal{H}^n(U)$ for open subsets $U$ of $\R^n$. If we want to emphasize that $A$ is a subset of $X$, we write $\mathcal{H}_{X}^n(A)$ instead of $\mathcal{H}^n(A)$.

We  state the  coarea inequality for Lipschitz functions, which is a consequence of \cite{Fed69}*{Theorem 2.10.25}.
\begin{thm}\label{thm:coarea:classical}
    Let $X$ be a metric space and let $g\colon X\to [0,\infty]$ be a Borel function. If $u\colon X\to \R$ is a $L$-Lipschitz function, $L\ge 1$, then 
        \begin{align*}
            \int\displaylimits^* \int_{u^{-1}(t)} g\, d\mathcal H^1dt \leq \frac{4}{\pi}L \int_X g \, d\mathcal H^2.
        \end{align*}
\end{thm}
Here, $\int^*$ denotes the upper integral, which is equal to the Lebesgue integral for measurable functions. In particular, $\int^*$ can be replaced by $\int$ in case of the right hand side being finite, see \cite{EIR22}*{Remark 2.12}. For a more general statement, we refer to \cite{EH21}.

The classical coarea formula for Sobolev functions below is stated in \cite{MalySwansonZiemer:Coarea} and attributed to Federer.
\begin{thm}\label{thm:coarea:Sobolev}
    Let $U\subset\R^2$ be open and let $g\colon U\to [0,\infty]$ be a Borel function that is finite almost everywhere. If $u\colon U\to \R$ is a continuous Sobolev function contained in $ W^{1,1}_{\text{loc}}(U)$, then
        \begin{align*}
            \int \int_{u^{-1}(t)}g\, d\mathcal H^1dt= \int_U g\cdot |\nabla u|\, d\mathcal H^2.
        \end{align*}
\end{thm}

The next density result follows from \cite{Fed69}*{Theorem 2.10.19(5)}. 

\begin{thm}\label{thm:regular}
    Let $X$ be a metric space of locally finite Hausdorff $n$-measure. Then there exists $E\subset X$, $\hm^n(E)=0$, so that
    \begin{eqnarray*} 
& & \limsup_{r \to 0}\frac{\hm^n(\bar{B}(x,r))}{\omega_n r^n}\leq 1 \quad \text{for every } x \in X \setminus E. 
    \end{eqnarray*}
\end{thm}

\subsection{Differentiability of Lipschitz maps and rectifiability} \label{sec:metric-diff}
Let $X$ be a metric space. A map $f\colon\R^n\to X$ is \emph{strongly metrically differentiable} at $x\in\R^n$ if there is a (unique) seminorm $\md f_x$ on $\R^n$ such that $$\lim_{(y,z)\to(x,x)}\frac{d(f(y), f(z))- \md f_x(y-z)}{|y-x| + |z-x|}= 0.$$ 

\begin{thm}[\cite{Kir94}*{Theorem 2}]\label{thm:LipschitzMetricDiff}
    If $f\colon\R^n\to X$ is Lipschitz, then $f$ is strongly metrically differentiable at $\hm^n$-almost every $x\in\R^n$.
\end{thm}

In particular, every rectifiable curve $\gamma\colon[a,b]\to X$ is strongly metrically differentiable at almost every $t\in[a,b]$, and the \emph{metric derivative} 
$$|\gamma'|(t)=\lim_{\varepsilon\searrow0}\frac{d(\gamma(t+\varepsilon), \gamma(t-\varepsilon))}{2\varepsilon}$$
exists and agrees with $\md\gamma_t(1)$ for almost every $t\in[a,b]$. The length of $\gamma$ may be expressed as 
$$
\ell(\gamma)=\int_a^b|\gamma'|(t)\,dt. 
$$
Note that if $\gamma$ is arc length parametrized, then $|\gamma'|(t)=1$ for almost every $a<t<b$.

A set $A\subset X$ is called \emph{countably $n$-rectifiable} if, up to a set of $\hm^n$-measure zero, $A$ can be covered by countably many Lipschitz images of subsets of $\R^n$. A set $U\subset X$ is \emph{purely n-unrectifiable} if $\hm^n(U\cap A)=0$ for every countably $n$-rectifiable set $A\subset X$. Every metric surface $X$ can be written as the disjoint union of a countably $2$-rectifiable part $\widehat X$ and a purely
$2$-unrectifiable part $X'$, see \cite{Bate22}*{Lemma 2.2}.

\subsection{Modulus of curve families}\label{section:modulus}
Let $X$ be a metric space of locally finite Hausdorff $n$-measure, let $\Gamma$ be a family of curves in $X$ and let $p\ge1$. A Borel function $g\colon X \to [0,\infty]$ is \emph{admissible} for $\Gamma$ if $\int_{\gamma}g\, ds\geq 1$ for all locally rectifiable curves $\gamma\in \Gamma$. We define the $p$-\emph{modulus} of $\Gamma$ as 
$$\mod_p \Gamma = \inf_g \int_X g^p \, d\mathcal H^n,$$
where the infimum is taken over all admissible functions $g$ for $\Gamma$. If there are no admissible functions for $\Gamma$ we set $\mod_p \Gamma = \infty$. A curve family $\Gamma_0$ is \emph{$p$-exceptional} if $\mod_p(\Gamma_0)=0$. A property is said to hold for \emph{$p$-almost every} curve in $\Gamma$ if it holds for every curve in $\Gamma\setminus\Gamma_0$ for some $p$-exceptional family $\Gamma_0\subset \Gamma$. 

The next lemma follows by noticing that $g=\infty \cdot \chi_A$ is an admissible function for the exceptional curve family $\Gamma_0$ associated to a null set $A \subset X$, see Definition \ref{def:exceptional}.  

\begin{lemma}[\cite{HKST:15}*{Lemma 5.2.15}]
    Let $A \subset X$ be a set of $\hm^n$-measure zero and let $p\ge 1$. Then, for $p$-almost every curve $\gamma$ in $X$ we have that $\hm^1(|\gamma|\cap A) = 0$.
\end{lemma}

\subsection{Metric Sobolev spaces}\label{section:Sobolev}
Let $f\colon X\to Y$ be a map between metric surfaces $X$ and $Y$. A Borel function $\rho \colon X\to [0,\infty]$ is an \textit{upper gradient} of $f$ if 
\begin{align}\label{ineq:upper_gradient}
    d_Y(f(x),f(y)) \leq \int_{\gamma} \rho \, ds
\end{align}
for all $x,y\in X$ and every rectifiable curve $\gamma$ in $X$ joining $x$ and $y$. If the \textit{upper gradient inequality} \eqref{ineq:upper_gradient} holds for $p$-almost every rectifiable curve $\gamma$ in $X$ joining $x$ and $y$ we call $\rho$ a \emph{$p$-weak upper gradient} of $f$. 

Let $L^p(X)$ ($L^p_{\text{loc}}(X)$) denote the \emph{space of (locally) $p$-integrable functions} from $X$ to $[-\infty,\infty]$. Here locally $p$-integrable means $p$-integrable on compact subsets.
The Sobolev space $N^{1,p}(X,Y)$ is the space of Borel maps $f \colon X \to Y$ with upper gradient $\rho \in L^p(X)$ such that $x \mapsto d_Y(y,f(x))$ is in $L^p(X)$ for some $y \in Y$. The space $N^{1,p}_\text{loc}(X, Y)$ is defined in the obvious manner. 

We refer to the monograph \cite{HKST:15} for more background on metric Sobolev spaces.

\subsection{Differentiability of Sobolev maps} \label{section:diff}
Let $X$ be a metric surface. The next result is a slight modification of 
\cite{LWarea}*{Proposition 4.3}, and follows from the proof given in \cite{LWarea}. The seminorm $\apmd h_z$ below is the \emph{approximate metric derivative} of $h$ at $z$. 

\begin{prop} \label{prop:appdiff}
Let $h \in N^{1,2}(D,X)$. For almost every $z \in D$ there exists a seminorm $\apmd h_z$ in $\R^2$ such that the following holds: There are pairwise disjoint compact sets $G_j \subset D$, $j=1,2,\ldots$, such that 
\begin{itemize} 
\item[(i)] $|G_0|_2=0$, where $G_0:=D \setminus \cup_{j=1}^\infty G_j$, 
\item[(ii)] if $j=1,2,\ldots$, then $h|_{G_j}$ is Lipschitz continuous and $\apmd h_z$ exists for every $z \in G_j$, and 
\item[(iii)] if $j=1,2,\ldots$, then for every $\varepsilon>0$ there exists $r_j(\varepsilon)>0$ such that 
$$ 
|d(h(z+v),h(z+w))-\apmd h_z(v-w)| \leq \varepsilon |v-w| 
$$ 
for every $z \in G_j$ and all $v,w \in \mathbb{R}^2$ with $|v|,|w| \leq r_j(\varepsilon)$ and such that $z+v,z+w \in G_j$.
\end{itemize}
\end{prop}

We will apply the following area formula. 

\begin{thm}[\cite{Kar07}*{Theorem 3.2}]\label{thm:area-formula}
    Let $h$ and $G_0$ be as in Proposition~\ref{prop:appdiff}. 
    If $E\subset D\setminus G_0$ is measurable, then 
    \begin{equation*}
       \int_E J(\apmd h_z) \,d A=\int_X N(x,h,E)\,d\mathcal H^2.
    \end{equation*}
\end{thm}

Here, $N(x,h,E)$ is the \emph{multiplicity}, i.e., the number of preimage points of $x$ in $E$ under $h$. Moreover, the \emph{Jacobian} $J(N_z)$ of a seminorm $N_z$ on $\R^2$ is zero if $N_z$ is not a norm and 
$$
J(N_z)={\pi}/{|\{x\in\R^2: N_z(y)\leq1\}|_2}
$$ 
otherwise. If we denote   
\begin{eqnarray*} 
   L_h(z)=\max\{\apmd h_z(v):|v|=1\}, \quad  
    l_h(z)=\min\{\apmd h_z(v):|v|=1\}, 
\end{eqnarray*}
then by \cite{MR23}*{Lemma 2.9} we have  
\begin{equation}\label{eq:disto}
2^{-1} L_h(z)l_h(z) \leq J(\apmd h_z) \leq 
2 L_h(z)l_h(z). 
\end{equation} 

\subsection{Uniformization of metric surfaces}\label{section:uniformization}
We now state a uniformization result which combines \cite{NR22}*{Theorem 1.2} and \cite{NR:21}*{Remark 7.2}. The dilatation bound \eqref{ineq:dilatationbound} does not appear in the statement, but follows from the proof of the \emph{weak quasiconformality} of $u$ given in \cite{NR:21}.

\begin{thm}\label{thm:uniformization}
    Let $X$ be a metric surface. For every $x_0 \in X$ there exist an open neighborhood $U \subset X$ and a continuous, surjective, proper and monotone map $u \in N^{1,2}(D,U)$ that satisfies $N(x,u,D)=1$ for almost every $x\in U$ and 
    \begin{equation} \label{ineq:dilatationbound} 
    L_u(z)^2 \leq \frac{4}{\pi} J(\apmd u_z) \quad \text{for almost every } z \in D. 
    \end{equation}
\end{thm}

Here $u$ is \emph{proper} if the preimage of each compact set is compact, and \emph{monotone} if the preimage of each point is connected. 

Let $X$ be a metric surface homeomorphic to $D$, and fix a map $u:D \to X$ satisfying the conclusions of Theorem \ref{thm:uniformization}.  We apply Proposition \ref{prop:appdiff} to obtain disjoint sets $G_j$ such that $|G_0|_2=0$ and $u|_{G_j}$ is Lipschitz continuous for every $j=1,2,\ldots$. 
\begin{definition} \label{def:singular}
We call $X^u_0:=u(G_0)$ the \emph{singular-set image} and $X \setminus X^u_0$ the \emph{rectifiable image} of $u$. 
\end{definition}

\section{Quadratic lower bound on area}\label{sec:auxiliary}
In this section, we assume that $X$ is a metric surface and prove an area bound which can be applied near $\gamma(t_0)$ when $t_0$ is a point of strong metric differentiability of a curve $\gamma$ in $X$. Namely, we establish the following result. Similar underlying ideas appear in \cite{NR22}*{Proof of Theorem 3.4}.
\begin{prop}\label{prop:inf-H^2-lower-bd}
    Let $k\ge 8$ be given and let $\gamma\colon[-\varepsilon_0,\varepsilon_0]\to X$ be a curve parametrized by arc length that satisfies 
    \begin{equation}\label{ineq:strong-diff}
        d(\gamma(t), \gamma(-t))\geq(2-k^{-1})t
    \end{equation}
    for every $0<t <\varepsilon_0$. There exists $\varepsilon_1>0$ 
    such that if $0<\varepsilon<\varepsilon_1$ and $\nicefrac{-\varepsilon}{2}<t_0<\nicefrac{\varepsilon}{2}$, then 
    $$\hm^2(B(\gamma(t_0),\delta))\geq\frac{\delta^2}{8},$$
    where  $\delta:=2\varepsilon k^{-1}$. 
\end{prop}

\begin{proof} 
We denote $x:=\gamma(0)$, $z_1:=\gamma(-\varepsilon_0)$, and $z_2:=\gamma(\varepsilon_0)$. Since $X$ is a connected topological surface, there are $0<\varepsilon_1<\nicefrac{\varepsilon_0}{2}$ and a path $\gamma_{z_2,z_1}$ in $X \setminus \bar{B}(x,\varepsilon_1)$ starting at $z_2$ and ending at $z_1$. We fix 
$0<\varepsilon < \varepsilon_1$ and $\nicefrac{-\varepsilon}{2}<t_0<\nicefrac{\varepsilon}{2}$. 

We denote $\delta:=2\varepsilon k^{-1}$, fix $\nicefrac{\delta}{2} < s < \delta$, and claim that 
\begin{equation} \label{ineq:bigcont}
\hm^1(S(\gamma(t_0),s)) \geq s.
\end{equation}

To prove \eqref{ineq:bigcont}, we first notice that since $\gamma$ is parametrized by arc length, Condition \eqref{ineq:strong-diff} gives 
\begin{align}\label{ineq:sepi}
    d(\gamma(t_0),\gamma(t))
    &\ge d(\gamma(-t),\gamma(t))
      - d(\gamma(-t),\gamma(-t/2))
      - d(\gamma(-t/2),\gamma(t_0))
      \notag\\
    &\ge (2-k^{-1})t
      -\frac{t}{2}
      -t
      =\left(\frac12-k^{-1}\right)t
      > s 
\end{align} 
for every $\varepsilon \leq t<\varepsilon_0$; here the last inequality holds since $k \geq 8$. In particular, $S(\gamma(t_0),s)$ separates $\gamma(t_0)$ and $\gamma(\varepsilon)$. By \cite{Moore62}*{IV Theorem 26}, $S(\gamma(t_0),s)$ contains a continuum $\Gamma(s)$ which also separates $\gamma(t_0)$ and $\gamma(\varepsilon)$. It follows that there is $t_0<b<\varepsilon$ such that $\gamma(b) \in \Gamma(s)$. 

We now consider the concatenated path  
$$
\eta:=\gamma|_{[\varepsilon,\varepsilon_0]} \ast \gamma_{z_2,z_1} \ast \gamma|_{[-\varepsilon_0,t_0]}, 
$$ 
which connects $\gamma(\varepsilon)$ and $\gamma(t_0)$. 
Since $\Gamma(s)$ separates $\gamma(\varepsilon)$ and $\gamma(t_0)$, $\eta$ intersects $\Gamma(s)$. A computation analogous to \eqref{ineq:sepi} shows that 
$$
d(\gamma(t_0),\gamma(t)) \geq s \quad \text{for every } 
-\varepsilon_0<t\leq -\varepsilon. 
$$
We conclude that $\gamma|_{[-\varepsilon_0,-\varepsilon]}$ does not intersect $\Gamma(s)$. But neither do $\gamma|_{[\varepsilon,\varepsilon_0]}$ (by \eqref{ineq:sepi}), nor $\gamma_{z_2,z_1}$ (by the choice of $s$), so there must be $-\varepsilon < a < t_0$ such that $\gamma(a) \in \Gamma(s)$. 
    
	As $\gamma$ is parametrized by arc length and $\gamma(a),\gamma(b) \in  S(\gamma(t_0),s)$, we get 
    \begin{equation}\label{eq:lem:t+-t-}
        b-a=b-t_0+t_0-a\geq 2s.
    \end{equation} 
    Condition \eqref{ineq:strong-diff}, the triangle inequality and parametrization of $\gamma$ by arc length furthermore imply that 
	\begin{equation}\label{eq:lem:bd-from-diff}
    \begin{split}
		(2-k^{-1})\varepsilon&\le d(\gamma(-\varepsilon),\gamma(\varepsilon))\\
        &\le d(\gamma(-\varepsilon),\gamma(a))+d(\gamma(a),\gamma(b))+d(\gamma(b),\gamma(\varepsilon))\\
		&\le a-(-\varepsilon)+d(\gamma(a),\gamma(b))+\varepsilon-b. 
    \end{split}
	\end{equation}
    Combining \eqref{eq:lem:t+-t-} and \eqref{eq:lem:bd-from-diff}, we obtain
   $$
    d(\gamma(a),\gamma(b))\ge b-a-k^{-1}\varepsilon\ge 2s-k^{-1}\varepsilon\ge s. 
   $$
    Therefore, since $\Gamma(s)$ is a continuum, we have 
     \begin{equation} \label{ineq:later} 
    \hm^1(S(\gamma(t_0),s)) \geq \hm^1(\Gamma(s)) \geq \diam \Gamma(s) \geq d(\gamma(a),\gamma(b)) \geq s,
    \end{equation}
    and our claim \eqref{ineq:bigcont} follows. 

    The coarea inequality (Theorem \ref{thm:coarea:classical}), applied to the distance function $y \mapsto d(\gamma(t_0),y)$, and \eqref{ineq:bigcont} now yield
    $$\hm^2(B(\gamma(t_0),\delta))\ge\frac{\pi}{4}\int_{\nicefrac{\delta}{2}}^{\delta}s\,ds\ge\frac{\delta^2}{8}. $$
    The proof is complete.  
\end{proof}

\section{Intrinsic proof of Theorem \ref{thm:main} and higher dimensions}\label{sec:proof:intrinsic}
This section is devoted to the intrinsic proof of the main theorem (Theorem \ref{thm:main}) and its higher dimensional version (Theorem \ref{thm:main-higher-dim}). 

\begin{proof}[Proof of Theorem \ref{thm:main}]
Let $X$ and $Y$ be metric surfaces, and let $f\colon X\to Y$ be area-preserving and $L$-Lipschitz, $L\ge1$. The upper bound of inequality \eqref{ineq:BLD} follows directly from the $L$-Lipschitz property of $f$. We are left to show the lower bound in \eqref{ineq:BLD}. By Theorem \ref{thm:regular} and the area-preservation of $f$, the set $$E:=\left\{x\in X:
\limsup_{r \to 0}\frac{\hm_Y^2(\bar{B}(f(x),r))}{\pi r^2}> 1\right\} $$
is of Hausdorff 2-measure zero. Here, measurability of $E$ holds as $E$ is the preimage of a zero set under a Lipschitz map. By definition, the family $\Gamma_0$ of all curves $\gamma$ in $X$ satisfying $\hm^1(|\gamma|\cap E)>0$ is exceptional. 

It suffices to prove the lower bound in inequality \eqref{ineq:BLD} for any rectifiable curve $\gamma$ that is not contained in $\Gamma_0$. Let $\gamma\colon[a,b]\to X$ be such a curve and, without loss of generality, assume that $\gamma$ is parametrized by arc length. By Theorem \ref{thm:LipschitzMetricDiff}, both $\gamma$ and $f\circ \gamma$ are strongly metrically differentiable at almost every $t\in[a,b]$. Thus, the set $J$ of all $t\in [a,b]$ with $\gamma(t)\in X\setminus E$ and both $\gamma$ and $f\circ \gamma$ are strongly metrically differentiable at $t$ is of full measure.
In particular, if we can show that there exists $C>0$ with
\begin{equation}\label{eq:proof-main-1:C}
    |(f\circ\gamma)'|(t)\ge C
\end{equation}
for every $t\in J$, then $$\ell(f\circ\gamma)=\int_a^b|(f\circ\gamma)'|(t)\,dt\ge C\,(b-a)=C\,\ell(\gamma).$$
This provides the desired lower bound in inequality \eqref{ineq:BLD}. We are left to show that \eqref{eq:proof-main-1:C} holds for a uniform constant $C=C(L)>0$.

Let $k:=10^4L^3$. We fix $t\in J$ and, without loss of generality, assume that $t=0$. As $\gamma$ is parametrized by arc length and strongly metrically differentiable at $t=0$, we get $|\gamma'|(0)=1$. 
In particular, by the definition of metric derivative, there exists $\varepsilon_0>0$ such that Condition \eqref{ineq:strong-diff} is satisfied for every $0<\varepsilon <\varepsilon_0$. 
Set $x=\gamma(0)$, and let $\varepsilon_1$ be the number in Proposition~\ref{prop:inf-H^2-lower-bd}. Choose $0<\varepsilon < \varepsilon_1$ so that
\begin{equation}\label{ineq:H^2-lower-bound}
	\hm^2(B(f(x),R))\leq 4 R^2
\end{equation}
holds for every $0< R <\varepsilon$. This is possible as $x$ is not contained in $E$. We set $$m:=10^3L^2\quad\text{and}\quad\delta:= 2\varepsilon k^{-1} L=2\varepsilon 10^{-4}L^{-2}.$$ 

Since $d(\gamma(-\varepsilon/2),\gamma(\varepsilon/2)) \geq (1-(2k)^{-1})\varepsilon$ by Condition \eqref{ineq:strong-diff}, we may choose $-\nicefrac{\varepsilon}{2}<t_1<\dots<t_m<\nicefrac{\varepsilon}{2}$ such that the balls
$$B_j:=B(\gamma(t_j),\delta/L)$$ are pairwise disjoint. The assumptions on $\gamma|_{(-\varepsilon_0,\varepsilon_0)}$ and on any ball $B_j$ are such that Proposition~\ref{prop:inf-H^2-lower-bd} may be applied to show that  
\begin{equation}\label{eq:proof-main-1:area-balls}
    \hm^2(B_j)\ge\frac{\delta^2}{8L^2}
\end{equation}
holds for every $j\in\{1,\dots,m\}$.
We define $\psi_\varepsilon:=f\circ\gamma|_{(-\varepsilon,\varepsilon)}$ and claim that
\begin{equation}\label{eq:proof-main-1:diam-image}
    \diam(|\psi_\varepsilon|)\ge \delta.
\end{equation}

Suppose towards a contradiction that inequality \eqref{eq:proof-main-1:diam-image} does not hold. Then, as $f$ is $L$-Lipschitz, it follows that 
$$B(f(x),3\delta)\supset f(\cup_{j=1}^m B_j).$$
Inequalities \eqref{ineq:H^2-lower-bound} and \eqref{eq:proof-main-1:area-balls}, and the area-preservation of $f$ give 
\begin{align*}
    36\delta^2&\ge\hm^2(B(f(x),3\delta))\ge \hm^2(\cup_{j=1}^m B_j)
    =\sum_{j=1}^m \hm^2(B_j)\ge m \frac{\delta^2}{8L^2}.
\end{align*}
This shows that $m\le 288 L^2$, a contradiction. Hence, \eqref{eq:proof-main-1:diam-image} is verified.

The preceding argument holds for all sufficiently small $\varepsilon>0$. For each such $\varepsilon$, choose points $a_\varepsilon,b_\varepsilon\in[-\varepsilon,\varepsilon]$ satisfying \[ d((f\circ\gamma)(a_\varepsilon),(f\circ\gamma)(b_\varepsilon)) =\operatorname{diam}(|\psi_\varepsilon|). \] Then, the definition of strong metric differentiability yields
\begin{align*}
    |(f\circ\gamma)'|(0)&=\lim_{\varepsilon\searrow0}\frac{d((f\circ\gamma)(a_\varepsilon),(f\circ\gamma)(b_\varepsilon))}{|a_\varepsilon|+|b_\varepsilon|}\\
    &\ge\lim_{\varepsilon\searrow0}\frac{\diam(|\psi_\varepsilon|)}{2\varepsilon}\ge \lim_{\varepsilon\searrow0}\frac{\delta}{2\varepsilon}= 10^{-4}L^{-1}.
\end{align*}
This concludes the proof of Theorem \ref{thm:main} after setting $C=10^{-4}L^{-1}$.
\end{proof}

Theorem \ref{thm:main-higher-dim} is implied by the following remark.
\begin{rmk} Let $X$ and $Y$ be metric spaces of locally finite Hausdorff $n$-measure, $n \geq 2$. 
    Note that, except for Proposition~\ref{prop:inf-H^2-lower-bd}, the proof of Theorem \ref{thm:main} presented in this section does not depend on planar topology. By assuming that the domain $X$ is lower Ahlfors $n$-regular with constant $c>0$, we circumvent the usage of Proposition~\ref{prop:inf-H^2-lower-bd}. 
    
    Namely, fix a point $x\in X\setminus E$, where $E$ is defined analogously as in the proof of Theorem \ref{thm:main}. Let $x=\gamma(t)$, where $\gamma$ and $t$ are chosen as above. Then, all of the above arguments hold after setting $k=1$ and choosing $m=50c^{-1}L^n$ and $\delta=200^{-1}c\varepsilon L^{-n}$. The statement of Theorem \ref{thm:main-higher-dim} now follows for $C=400^{-1}c L^{-n}$.  

    The above argument can be generalized to get a stronger statement of Theorem \ref{thm:main-higher-dim} where instead of the global lower Ahlfors $n$-regularity \eqref{eq:lower-Ahlfors} one assumes that there is $c>0$ such that almost every $x \in X$ has an open neighborhood on which \eqref{eq:lower-Ahlfors} holds. 
\end{rmk}

\section{Singular parts of metric surfaces} \label{sec:proof-thm-unrectifiable}
The goal of this section is to establish Theorem \ref{thm:unrectifiable} as well as the essential independence of the singular-set images, Theorem \ref{thm:singular}. 

Let $X$ be a metric surface. By covering $X$ with open subsets satisfying the assumptions in Theorem \ref{thm:uniformization} if necessary, we may assume that there is a uniformization map $u:D \to X$ satisfying the conclusions of Theorem~\ref{thm:uniformization}. Let $X^u_0$ be the singular-set image of $u$ in Definition \ref{def:singular}. 
The proof of Theorem \ref{thm:unrectifiable} is carried out by investigating \emph{curve-regular} points.  

\begin{definition} 
A point $x \in X$ is \emph{curve-regular} if there exists an arc length parametrized curve $\gamma\colon [a,b]\to X$ such that $x=\gamma(t)$ for some $t\in(a,b)$ and $\gamma$ is strongly metrically differentiable at $t$ with $|\gamma'|(t)=1$. We denote the set of curve-regular points by $X_1$.
\end{definition} 

The next result follows from Proposition \ref{prop:inf-H^2-lower-bd} and its proof; see \eqref{ineq:later} for the proof of Part (i).   

\begin{corollary}\label{cor:value-metr-diff}
    For every $x\in X_1$ there is $s_x>0$ such that the following properties hold for every $0<s<s_x$: 
    
    \begin{itemize}
    \item[(i)] $S(x,s)$ contains a continuum $\Gamma(s)=\Gamma(x,s)$ with $\diam \Gamma(x,s)\ge s$. 
    \item[(ii)] we have $$\hm^2(B(x,s))\ge \frac{\pi}{8}s^2. $$  
    \end{itemize}
\end{corollary}

Now suppose that $\gamma\colon[a,b]\to X$ is a rectifiable curve satisfying the condition of Theorem \ref{thm:unrectifiable}, i.e., 
$$ 
\hm^1(|\gamma| \cap X^u_0)>0. 
$$
We may assume without loss of generality that $\gamma$ is arc length parametrized. Since $\gamma$ is strongly metrically differentiable at almost every $t\in[a,b]$ by Theorem \ref{thm:LipschitzMetricDiff}, we furthermore have 
$$
\hm^1(|\gamma| \cap (X^u_0 \cap X_1))>0. 
$$
Therefore, Theorem \ref{thm:unrectifiable} follows after we establish the next proposition. 

\begin{prop}\label{prop:F} 
We have $\hm^2(F)=0$, where $F:=X^u_0 \cap X_1$. 
\end{prop}

\begin{proof}
    Suppose towards a contradiction that $\hm^2(F)>0$. Then $F\subset X^u_0$ has a compact subset $F'\subset F$ satisfying $\hm^2(F')>0$ and 
$$
N(x,u,D)=1 \quad \text{for every } x \in F'; 
$$
here we use the Borel regularity of $\hm^2$ and Theorem \ref{thm:uniformization}.

    Given $\varepsilon>0$, by outer regularity of Lebesgue measure and the fact that 
    $$
    u^{-1}(F) \subset G_0, \quad \text{hence} \quad |u^{-1}(F)|_2=|G_0|_2=0, 
    $$ 
    there exists an open set $U\subset D$ with $u^{-1}(F')\subset U$ and $|U|_2<\varepsilon$. The set $V:=D\setminus U$ is closed in $D$, and, by continuity and properness of $u$, the set $u(V)$ is closed in $X$. Compactness of $F'$ and the $5r$-covering lemma, see e.g.\ \cite{Hei01}*{Theorem 1.3}, now imply the following: 

    There exists a family of pairwise disjoint balls $B_j=B(x_j,r_j)\subset X\setminus u(V)$ centered at $x_j\in F'$ of radius $0<r_j <\min\{\varepsilon,s_{x_j}\}$, where $s_{x_j}>0$ is the constant in  Corollary \ref{cor:value-metr-diff}, such that the balls $5Bj:=B(x_j,5r_j)$ cover $F'$ and
    \begin{equation}\label{ineq:Lebesgue}
        \big|u^{-1}\big(\bigcup_{j} B_j\big)\big|_2\le 
        |D\setminus V|_2=|U|_2<\varepsilon.
    \end{equation}
    
    Fix $j$. By \cite{HKST:15}*{Theorem 7.1.20}, and the fact that $u\in N^{1,2}(D,X)$, the function $u_j\colon D\to\R$ defined by $u_j(z):=d(x_j,u(z))$ is in the classical Sobolev space $W^{1,2}(D)$. 
    By the coarea formula for Sobolev functions (Theorem \ref{thm:coarea:Sobolev}), there exists a set of full measure $J_j\subset(\nicefrac{r_j}{2},r_j)$ such that the level sets $u_j^{-1}(r)=u^{-1}(S(x_j,r))$ satisfy $\hm^1(u_j^{-1}(r))<\infty$ for every $r\in J_j$. 
    
    Choose any $r\in J_j$, and let $\Gamma(x_j,r) \subset S(x_j,r)$ be the continuum in  Corollary~\ref{cor:value-metr-diff}. Properness and monotonicity of $u$ imply that $u^{-1}(\Gamma(x_j,r))$ is compact and connected. By \cite{RR19}*{Proposition 5.1}, there exists a 1-Lipschitz curve $\psi_{j,r}\colon[a,b]\to D$ which satisfies 
    \begin{equation} \label{eq:surji} 
    |u\circ\psi_{j,r}|=\Gamma(x_j,r) 
    \end{equation} 
    and $\psi_{j,r}^{-1}(z)$ contains at most two points for $\hm^1$-almost every $z\in u_j^{-1}(r)$.

    Let $\rho\in L^2(D)$ be an upper gradient of $u$. Since the distance function is 1-Lipschitz, we get $|\nabla u_j|\le \rho$ almost everywhere in $D$. The assumptions on $\psi_{j,r}$ are such that 
    \begin{align*}
         2\int_{u_j^{-1}(t)}\rho\, d\mathcal H^1\ge \int_{\psi_{j,r}}\rho\, ds\ge \ell(u\circ\psi_{j,r})\ge \diam(|u\circ\psi_{j,r}|)\ge r; 
    \end{align*}
    here the last inequality follows from Corollary \ref{cor:value-metr-diff} and \eqref{eq:surji}. 
    As this is true for almost every $\nicefrac{r_j}{2} <r< r_j$,
    the coarea formula for Sobolev functions (Theorem \ref{thm:coarea:Sobolev}) applied to $u_j$ and $g=\rho$ now yields 
        \begin{align*}
            \frac{r_j^2}{8}\le\int_{\nicefrac{r_j}{2}}^{r_j} \int_{u_j^{-1}(r)}\rho\, d\mathcal H^1dr\le \int_{u^{-1}(B_j)} \rho\cdot|\nabla u_j|\, d\mathcal H^2\le \int_{u^{-1}(B_j)} \rho^2\, d\mathcal H^2. 
        \end{align*}
    Note that the balls $5B_j$ all satisfy $\diam 5B_j<10\varepsilon$ and cover $F'$. Thus, by definition of the Hausdorff content $\hm^2_{10\varepsilon}$, we get
    \begin{equation}\label{ineq:hmF'}
    \begin{split}
        \hm^2_{10\varepsilon}(F')\le \sum_j(5r_j)^2\le200\sum_j\frac{r_j^2}{8}&\le 200\sum_j \int_{u^{-1}(B_j)} \rho^2\, d\mathcal H^2\\&=200\int_{u^{-1}(\bigcup_jB_j)} \rho^2\, d\mathcal H^2,
    \end{split}
    \end{equation}
    where the last equality follows as the balls $B_j$ are disjoint. By applying \eqref{ineq:Lebesgue}, \eqref{ineq:hmF'}, and the absolute continuity of integrals, we obtain $\hm^2_{10\varepsilon}(F')\to 0$ as $\varepsilon\to 0$. This contradicts $\hm^2(F')>0$, and thus, $\hm^2(F)>0$.
\end{proof}

We note that slight modification of the proof given above gives the following generalization of Theorem \ref{thm:unrectifiable}. 

\begin{thm}\label{thm:unrectifiable:general}
    Let $X$ be a metric surface, let $M$ be a Riemannian 2-manifold and let $u\in\Nloc(M,X)$ be continuous, surjective, proper and monotone. If $A\subset X$ satisfies $|u^{-1}(A)|_2=0$, then the family of all curves $\gamma$ in $X$ with $\hm^1(|\gamma|\cap A)>0$ is exceptional.
\end{thm}

We end this section by proving Theorem \ref{thm:singular}, i.e., that the singular-set image of $u$ is essentially independent of $u$. 
The proof will depend on the following lemma.

\begin{lemma}\label{lem:singular}
    Almost every point in $X\setminus X_0^u$ is contained in $X_1$.
\end{lemma}

\begin{proof}
    We fix a square 
$$
Q=Q(z_0,r)=z_0+[0,r]^2 \subset D, 
$$
and note that it suffices to show that almost every point in 
$u(G_j\cap Q)$, $j=1,2,\ldots$, is in $X_1$; recall the sets $G_j$ in Proposition \ref{prop:appdiff}. 

The Sobolev map $u$ is absolutely continuous on every horizontal line segment $I_s:=z_0+is+[0,r] \subset Q$, $s \in [0,r] \setminus E$, where 
$|E|_1=0$. Here, for convenience of notation, we consider $D$ as a subset of the complex plane $\mathbb{C}$. For $s\in[0,r]\setminus E$, we denote  
$$
\gamma_s:[0,r] \to Q, \quad \gamma_s(t):=z_0+t+is
$$
and consider the set $G_j'$ of points  $z=\gamma_s(t)\in Q \cap G_j$ that satisfy the following conditions : 
\begin{itemize} 
\item[(i)] $J(\apmd u_{z})>0$, 
\item[(ii)] $s \in [0,r] \setminus E$ and $u \circ \gamma_s$ is strongly metrically differentiable at $t$.    
\item[(iii)] $z$ is a density point of $G_j \cap I_s$ in $I_s$. 
\end{itemize}
Note that if $Z$ is the set of points in $G_j \cap Q$ for which condition (i) does not hold, then $\hm^2(u(Z))=0$ by the area formula (Theorem \ref{thm:area-formula}). On the other hand, the set of points in $G_j \cap Q$ for which one of conditions (ii) or (iii) does not hold is of Lebesgue measure zero by Fubini's theorem. Therefore, by the area formula we have 
$\hm^2(u(G_j \cap Q \setminus G_j'))=0$. We conclude that the theorem follows if we can prove that $\hm^2(u(G_j') \setminus X_1)=0$. 

The set $G_j'$ is defined such that 
\begin{equation} \label{eq:soki}
|(u\circ \gamma_s)'|(t)>0 \quad \text{for every } z=\gamma_s(t) \in G_j'. 
\end{equation} 
Indeed, condition (i) and \eqref{eq:disto} imply that the minimal dilatation of $u$ at $z$ satisfies 
$l_u(z) >0$, while conditions (ii) and (iii) together with Proposition \ref{prop:appdiff} imply that $|(u\circ \gamma_s)'|$ exists at $t$ and is bounded from below by $l_u(z)$. 

We fix $s \in [0,r] \setminus E$, and let $\gamma^u_s:[0,\ell_s] \to u(\gamma_s([0,r]))$ be the arc length parametrization of $u \circ \gamma_s$. By Theorem \ref{thm:LipschitzMetricDiff}, the set 
$W_s:=u(\gamma_s([0,r])) \setminus X_1$ of non-curve-regular points 
satisfies $\hm^1(W_s)=0$. Therefore, by \eqref{eq:soki} and the change of variables formula, 
\begin{equation} \label{eq:prefubi}
|(G_j' \cap I_s) \setminus u^{-1}(X_1)|_1=0.  
\end{equation} 
Combining \eqref{eq:prefubi} and Fubini's theorem, we conclude that 
$$
|G_j' \setminus u^{-1}(X_1)|_2=0. 
$$
A final application of the area formula then shows that $\hm^2(u(G_j') \setminus X_1)=0$, as desired. The proof is complete. 
\end{proof}

\begin{proof}[Proof of Theorem \ref{thm:singular}]
By covering $X$ with open subsets that satisfy the assumptions of Theorem \ref{thm:uniformization} if necessary, we may assume that all the uniformization maps onto $X$ are from $D$. We fix such a map $u_0:D \to X$, and set $X_0:=X_0^{u_0}$. Let $u$ be another uniformization map. Proposition \ref{prop:F} applied to $u$ and Lemma \ref{lem:singular} applied to $u_0$ yield   
$$ 
\hm^2(X_0^{u} \setminus X_0) \leq \hm^2(X_0^{u} \cap X_1)+ 
\hm^2(X \setminus (X_0 \cup X_1))=0. 
$$
The complementary identity $\hm^2(X_0 \setminus X_0^{u})=0$ follows by changing the roles of $u_0$ and $u$. The proof is complete. 
\end{proof}

\section{Proof of Theorem \ref{thm:main} via uniformization of metric surfaces}\label{sec:proof-uniformization} 
In this section, we provide a proof of Theorem \ref{thm:main} that depends on uniformization of metric surfaces. Let $X$ and $Y$ be metric surfaces and $f\colon X\to Y$ an area-preserving $L$-Lipschitz map, $L\ge1$. 
As before, we may assume that there is a uniformization map $u\colon D\to X$ as in Theorem \ref{thm:uniformization}. Let $X_0=u(G_0)$ be the corresponding singular-set image in Definition \ref{def:singular}, and let $G_j$, $j=0,1,2,\ldots$ be the sets in Proposition \ref{prop:appdiff}. 

Theorem \ref{thm:main} follows by combining Theorem \ref{thm:unrectifiable} with the next result. 

\begin{thm}\label{thm:rectifiable}
We have 
    \begin{equation}\label{ineq:BLD-rectifiable2}
        \frac{\pi}{8L}\,\ell(\gamma)\leq\ell(f\circ\gamma)\leq L\,\ell(\gamma) 
    \end{equation} 
    for almost every rectifiable curve $\gamma$ in $X$ for which $\hm^1(|\gamma|\cap X_0)=0$.
\end{thm}

\begin{rmk}\label{rmk:C(1)=1}
    For $L=1$, Theorem \ref{thm:rectifiable} follows from \cite{CreutzSoultanis}*{Proposition 4.1}, which holds with constant $C=1$ and for all countably $n$-rectifiable metric spaces $X,Y$. Thus, by making use of this instead of the proof given below, we may set $C(1)=1$ in Theorem \ref{thm:rectifiable} and in Theorem \ref{thm:main}.
\end{rmk}

The rest of this section is devoted to the proof of Theorem \ref{thm:rectifiable}. The upper bound in \eqref{ineq:BLD-rectifiable2} follows directly from the $L$-Lipschitz property of $f$. We are left to show the lower bound. 
    
Since $f$ is Lipschitz continuous, the map $h:=f\circ u$ is in $N^{1,2}(D,Y)$. An application of Theorem \ref{thm:unrectifiable:general} then shows that we lose no generality by assuming that, in addition to $u$, also $h$ satisfies the conclusions of Proposition~\ref{prop:appdiff} with the sets $G_j$ above. 

We define the following subsets of $X$: 
\begin{eqnarray*}
X'_0&:=&\{x: \, N(x,u,D)>1)\}, \\
X''_0&:=&\{x \notin X_0 \cup X'_0: \, J(\apmd u(u^{-1}(x)))=0\}, \quad \text{and } \\  
X'''_0&:=&\{x \notin X_0 \cup X'_0 \cup X''_0: \, 
J(\apmd u(u^{-1}(x))) \neq J(\apmd h(u^{-1}(x))) \}. 
\end{eqnarray*}
Recall that $X_0'$ is a null set by Theorem \ref{thm:uniformization}. Moreover, the area formula (Theorem \ref{thm:area-formula}) holds outside the singular-set image $X_0=u(G_0)$. Therefore $X_0''$ is also a null set, and
$$
\hm^2(X'_0 \cup X''_0)=0. 
$$
We claim that also 
\begin{equation} \label{eq:'''} 
\hm^2(X'''_0)=0. 
\end{equation}
Indeed, we have $N(y,f,D)=1$ for $\hm^2$-almost every $y \in Y$ by \cite{MeierNtalampekos}*{Lemma 3.2}. Since also $N(x,u,D)=1$ for $\hm^2$-almost every $x \in X$ and $f$ is area-preserving, $N(y,h,D)=1$ for every $y \in Y \setminus Y_0'$, where $\hm^2(Y_0')=0$. 

We now apply the area formula (Theorem \ref{thm:area-formula}) to both $u$ and $h$ on measurable sets 
$$
E \subset F:= D \setminus (u^{-1}(X_0 \cup X_0' \cup X_0'') \cup h^{-1}(Y_0')), 
$$ 
together with the area-preserving property of $f$, to conclude that 
$$
\int_{E} J(\apmd u(z))\, dA = \hm^2(u(E))=\hm^2(h(E))=
\int_{E} J(\apmd h(z))\, dA. 
$$
It follows that $J(\apmd u(z))=J(\apmd h(z))$ for almost every  
$z \in F$. Moreover, since the restriction of $u$ to $F \subset D \setminus G_0$ maps sets of zero Lebesgue measure to sets of zero Hausdorff $2$-measure, \eqref{eq:'''} holds.

We conclude that it suffices to prove the lower bound in \eqref{ineq:BLD-rectifiable2} for rectifiable curves $\gamma$ for which 
\begin{equation} \label{eq:orez}
\hm^1(|\gamma|\cap (X_0 \cup X'_0 \cup X''_0 \cup X'''_0))=0. 
\end{equation}
We fix such a curve $\gamma$. We may assume that $\gamma$ is arc length parametrized. Theorem \ref{thm:rectifiable} now follows if we can prove that 
\begin{equation} \label{ineq:derbound}
|(f\circ\gamma)'|(t) \geq \frac{\pi}{8L} \quad \text{for almost every } 0<t <\ell(\gamma). 
\end{equation}
    
    We denote 
    $$
    X_j:=u(G_j) \setminus (X'_0 \cup X''_0 \cup X'''_0), \quad j=1,2,\ldots. 
    $$
    By \eqref{eq:orez} and the fact that the preimage of a set of $\hm^1$-measure zero under an arc length parametrized curve is negligible, almost every $0<t_0<\ell(\gamma)$ 
    satisfies 
    \begin{equation} \label{eq:density}
    t_0 \in F_j:=\gamma^{-1}(X_j)\text{ is a density point of } F_j \text{ for some } j=1,2,\ldots;  
    \end{equation}
    note that $F_j \subset [0,\ell(\gamma)]$. 
    
    Recall that by Theorem \ref{thm:LipschitzMetricDiff}, both $\gamma$ and $f\circ \gamma$ are strongly metrically differentiable at almost every point. We fix $t_0$ such that \eqref{eq:density} holds and such that both $\gamma$ and $f \circ \gamma$ are strongly metrically differentiable at $t_0$. By \eqref{eq:density}, we can fix a sequence of points $t_m \in F_j$, $t_m \neq t_0$, such that $t_m \to t_0$ as $m \to \infty$. We denote $x_m:=\gamma(t_m)$ and $z_m:=u^{-1}(x_m)$; recall that $z_m$ is a single point by the definition of $X_j$. Proposition \ref{prop:appdiff} shows that 
    \begin{eqnarray} \label{ineq:derivest1}
& & \liminf_{m \to \infty} \frac{d_Y(f(x_m),f(x_0))}{|z_m-z_0|} \geq l_h(z_0) \quad \text{and } \\ \label{ineq:derivesti2} 
& & \limsup_{m \to \infty} \frac{d_X(x_m,x_0)}{|z_m-z_0|} \leq L_u(z_0); 
    \end{eqnarray}
here $l_h$ and $L_u$ are the minimal, resp., maximal dilatations of the approximate metric derivatives, defined before \eqref{eq:disto}. 

By the strong metric differentiability of $\gamma$ and $f\circ \gamma$ at $t_0$ and since ${|\gamma'|(t_0)=1}$, we have 
$$
|(f\circ \gamma)'|(t_0) = \lim_{m \to \infty} \frac{d_Y(f(x_m),f(x_0))}{d_X(x_m,x_0)} \geq \frac{l_h(z_0)}{L_u(z_0)}, 
$$
where the inequality is a consequence of \eqref{ineq:derivest1} and \eqref{ineq:derivesti2}. Therefore, the desired estimate \eqref{ineq:derbound} follows if we can show that 
\begin{equation} \label{ineq:hubound} 
\frac{l_h(z_0)}{L_u(z_0)} \geq \frac{\pi}{8L}. 
\end{equation} 
To prove \eqref{ineq:hubound}, we note that $L_h(z_0) \leq L_u(z_0) L$, where $L$ is the Lipschitz constant of $f$. Therefore, the dilatation estimate \eqref{eq:disto} yields 
\begin{equation} \label{ineq:lastdilatationbound} 
\frac{l_h(z_0)}{L_u(z_0)} \geq \frac{J(\apmd h(z_0))}{2L_u(z_0)^2L}=\frac{J(\apmd u(z_0))}{2L_u(z_0)^2L}; 
\end{equation}
here the equality follows from the definitions of $X_j$ and $X'''_0$, and our choice of $t_0$ in \eqref{eq:density}. Combining \eqref{ineq:lastdilatationbound} with \eqref{ineq:dilatationbound} gives \eqref{ineq:hubound}. The proof of Theorem \ref{thm:rectifiable} is complete.


\end{document}